\documentclass[reqno]{amsart}
\usepackage{amsthm,amsmath,amssymb}
\usepackage{stmaryrd,mathrsfs}

\newcommand{\ud}[0]{\,\mathrm{d}}

\newcommand{\dist}[0]{\operatorname{dist}}

\newcommand{\abs}[1]{|#1|}
\newcommand{\Babs}[1]{\Big|#1\Big|}

\newcommand{\Norm}[2]{\|#1\|_{#2}}
\newcommand{\BNorm}[2]{\Big\|#1\Big\|_{#2}}
\newcommand{\pair}[2]{\langle #1,#2 \rangle}

\newcommand{\ave}[1]{\langle #1\rangle}

\newcommand{\bddlin}[0]{\mathscr{L}}

\newcommand{\BMO}[0]{\operatorname{BMO}}
\newcommand{\supp}[0]{\operatorname{supp}}

\newcommand{\R}{\mathbb{R}}
\newcommand{\C}{\mathbb{C}}

\newcommand{\Z}{\mathbb{Z}}

\newcommand{\prob}[0]{\mathbb{P}}
\newcommand{\Exp}[0]{\mathbb{E}}

\newcommand{\radem}[0]{\varepsilon}

\newcommand{\ontop}[2]{\begin{smallmatrix}#1\\#2\end{smallmatrix}}

\numberwithin{equation}{section}

\makeatletter
  \let\c@equation\c@subsection
\makeatother

\theoremstyle{plain}

\newtheorem{theorem}[subsection]{Theorem}

\newtheorem{corollary}[subsection]{Corollary}
\newtheorem{lemma}[subsection]{Lemma}

\theoremstyle{definition}
\newtheorem{definition}[subsection]{Definition}

\theoremstyle{remark}

\begin{document}

\title[Weighted, vector-valued BMO functions]{Wavelet expansions for weighted, vector-valued BMO functions}

\author[T.~Hyt\"onen \and O.~Salinas \and B.~Viviani]{Tuomas Hyt\"onen \and Oscar Salinas \and Beatriz Viviani}

\address{T.H.: Department of Mathematics and Statistics, University of Helsinki, Gustaf H\"allstr\"omin katu 2b, FI-00014 Helsinki, Finland}
\email{tuomas.hytonen@helsinki.fi}

\address{O.S. \& B.V.: Instituto de Matem\'atica Aplicada del Litoral, G\"uemes 3450, 3000 Santa Fe, Rep\'ublica Argentina}
\email{salinas@santafe-conicet.gov.ar}
\email{viviani@santafe-conicet.gov.ar}

\date{\today}



\begin{abstract}
We introduce a scale of weighted Carleson norms, which depend on an integrability parameter \(p\), where \(p=2\) corresponds to the classical Carleson measure condition. Relations between the weighed BMO norm of a vector-valued function \(f:\R\to X\), and the Carleson norm of the sequence of its wavelet coefficients, are established. These extend the results of Harboure--Salinas--Viviani, also in the scalar-valued case when \(p\neq 2\).
\end{abstract}

\maketitle


\section{Introduction}\label{intro}

Given a positive non-decreasing function \(\varrho\) on \((0,\infty)\) (so-called growth function), a weight \(w\) on \(\R\), and a Banach space \(X\), the function space \(\BMO_{\varrho}(w;X)\) consists of those locally Bochner integrable \(f:\R\to X\) for which the norm
\begin{equation*}
   \Norm{f}{\BMO_{\varrho}(w;X)}
   :=\sup_I\frac{1}{w(I)\varrho(\abs{I})}\int_I\Norm{f(x)-\ave{f}_I}{X}\ud x
\end{equation*}
is finite. Here and below, \(\sup_I\) refers to supremum over all finite intervals \(I\subset\R\), and we use the abbreviations \(w(I):=\int_I w(x)\ud x\) and \(\ave{f}_I:=\abs{I}^{-1}\int_I f(x)\ud x\). This is a natural vector-valued generalization of the space \(\BMO_{\varrho}(w):=\BMO_{\varrho}(w;\C)\), which has been recently studied in \cite{HSV,HSV:2004,HSV:2007,Morvidone}.

Recall that \(\psi\in L^2(\R)\) is called an orthonormal wavelet if the functions \(\psi_J(x):=\abs{J}^{-1/2}\psi\big(\abs{J}^{-1}(x-\inf J)\big)\) form an orthonormal basis of \(L^2(\R)\) when \(J\) runs through the set \(\mathscr{D}\) of all dyadic intervals \(J=2^j[k,k+1)\), \(j,k\in\Z\). In~\cite{HSV}, the wavelet coefficients \(\pair{\psi_J}{f}:=\int\psi_J(x)f(x)\ud x\) of \(f\in\BMO_{\varrho}(w)\) were studied, and it was shown --- under appropriate conditions on \(\varrho\), \(w\), and \(\psi\) --- that the norm \(\Norm{f}{\BMO_{\varrho}(w)}\) dominates the following Carleson measure norm, where \(a_J=\pair{\psi_J}{f}\):
\begin{equation*}
\begin{split}
  \Norm{\{a_J\}_{J\in\mathscr{D}}}{C_{\varrho}(w)}
  :=& \sup_I\frac{1}{\varrho(\abs{I})}\Big(\frac{1}{w(I)}\sum_{\ontop{J\in\mathscr{D}}{J\subseteq I}}
     \abs{a_J}^2\frac{\abs{J}}{w(J)}\Big)^{1/2},  \\
  =& \sup_I\frac{1}{\varrho(\abs{I})}\Big(\frac{1}{w(I)}\int_I\Exp\Babs{\sum_{\ontop{J\in\mathscr{D}}{J\subseteq I}}
       \radem_J a_J\Big(\frac{\abs{J}}{w(J)}\Big)^{1/2}\frac{1_J(x)}{\abs{J}^{1/2}}}^2\ud x\Big)^{1/2}.
\end{split}
\end{equation*}
Conversely, the finiteness of this norm implies that the series \(\sum_{J\in\mathscr{D}}a_J\psi_J\) converges, in a suitable sense, to a function in \(\BMO_{\varrho}(w)\) whose norm is controlled by $\Norm{\{a_J\}_{J\in\mathscr{D}}}{C_{\varrho}(w)}$.

On the right side, the \(\radem_J\) designate independent random signs with probability distribution \(\prob(\radem_J=+1)=\prob(\radem_J=-1)=\frac{1}{2}\), and \(\Exp\) is the mathematical expectation. The equality, which is completely elementary for scalar coefficients \(a_J\) (or even Hilbert space -valued ones), no longer holds for \(a_J\in X\), when \(X\) is a more general Banach space.

In our situation, following the experience from other vector-valued problems (e.g., \cite{Bourgain,CPSW,H:Wavelets}), we take the right side as the definition of the vector-valued Carleson norm. In fact, we define a one-parameter scale of such norms by setting
\begin{equation*}
\begin{split}
  &\Norm{\{a_J\}_{J\in\mathscr{D}}}{C_{\varrho}^p(w;X)}  \\
  &:= \sup_I\frac{1}{\varrho(\abs{I})}\Big(\frac{1}{w(I)}\int_I\Exp\BNorm{\sum_{\ontop{J\in\mathscr{D}}{J\subseteq I}}
       \radem_J a_J\Big(\frac{\abs{J}}{w(J)}\Big)^{1/p'}\frac{1_J(x)}{\abs{J}^{1/2}}}{X}^p\ud x\Big)^{1/p}.
\end{split}
\end{equation*}
As usual in vector-valued harmonic analysis, we require that our Banach space have the unconditionality property for martingale difference sequences (UMD). In the present paper, this assumption will never be used directly, but rather through a number of earlier results which have been established in this class of spaces and which will be recalled below.
See e.g.~\cite{Bourgain,CPSW} for more on this notion.

Let us now assume that the weight \(w\) is in the Muckenhoupt class \(A_q\) for some fixed \(q\in(1,2)\), i.e.,
\begin{equation*}
  \sup_I\frac{1}{\abs{I}}\int_I w(x)\ud x
  \Big(\frac{1}{\abs{I}}\int_I w^{-1/(q-1)}(x)\ud x\Big)^{q-1}\leq C.
\end{equation*}
Let the growth function be such that \(\int_1^{\infty}\varrho(s)s^{q-3}\ud s<\infty\), and set
\begin{equation*}
  \eta(t):=t^{2-q}\int_t^{\infty}\frac{\varrho(s)}{s^{3-q}}\ud s
  =\int_1^{\infty}\frac{\varrho(tu)}{u^{3-q}}\ud u.
\end{equation*}

Note that \(\eta(t)\geq\varrho(t)\), since \(\varrho\) is non-decreasing. If \(\varrho\) has the property that
\begin{equation*}
  \varrho(ut)\leq Cu^{\alpha}\varrho(t)
\end{equation*}
for some \(\alpha<2-q\) and all \(t>0\), \(u>1\) (so-called upper-type \(\alpha\)), then conversely \(\eta(t)\lesssim\varrho(t)\), so that \(\eta\) and \(\varrho\) are comparable, but we do not necessarily assume this. However, we do assume that \(\varrho\) has some upper-type \(\alpha<\infty\), which is equivalent to the doubling property \(\varrho(2t)\leq C\varrho(t)\).

Finally, we say that a function \(\phi\) on \(\R\) is of class \(\Psi^u_v\) if
\begin{equation*}
  \abs{\phi(x)}\leq C(1+\abs{x})^{-u},\qquad
  \abs{\phi'(x)}\leq C(1+\abs{x})^{-v}.
\end{equation*}
Under these assumptions, we have the following results:

\begin{theorem}\label{thm:BMOtoCar}
Let \(\psi\in\Psi_{1+\radem}^{2+\radem}\), \(\radem>0\), be an orthonormal wavelet.
If \(f\in\BMO_{\varrho}(w;X)\) and \(p\in(1,q']\), then \(\{\pair{\psi_J}{f}\}_{J\in\mathscr{D}}\in C_{\eta}^p(w;X)\) and 
\begin{equation*}
  \Norm{\{\pair{\psi_J}{f}\}_{J\in\mathscr{D}}}{C_{\eta}^p(w;X)}
  \lesssim\Norm{f}{\BMO_{\varrho}(w;X)}.
\end{equation*}
\end{theorem}

\begin{theorem}\label{thm:CarToBMO}
Let \(\psi\in\Psi_{2}^{2+\radem}\), \(\radem>0\), be an orthonormal wavelet.
If \(\{a_J\}_{J\in\mathscr{D}}\in C_{\varrho}^p(w;X)\) for some \(p\in(1,\infty)\), then the series
\(
  \sum_{J\in\mathscr{D}}a_J\psi_J
\)
converges to a function \(f\in\BMO_{\eta}(w;X)\) in the following sense: For every interval \(I\subset\R\), there are ``renormalization constants'' \(c_J\in\C\), $\xi\in X$ such that
\begin{equation*}
  \sum_{J\in\mathscr{D}}a_J(\psi_J-c_J)|_I
\end{equation*}
converges to \(f|_I-\xi\), unconditionally in \(L^s(I;X)\) for some \(s>1\). Here \(f|_I\) means the restriction of \(f\) on \(I\). Moreover,
\begin{equation*}
  \Norm{f}{\BMO_{\eta}(w;X)}
  \lesssim\Norm{\{a_J\}_{J\in\mathscr{D}}}{C_{\varrho}^p(w;X)}
\end{equation*}
\end{theorem}

Under the additional assumption that \(\varrho\) be of upper type \(\alpha<2-q\), the growth function \(\eta\) in the theorems may be replaced by \(\varrho\), so that the two results establish a kind of norm equivalence.
The appearance of the ``renormalization constants'' may be seen as a reflection of the fact that constant functions have a vanishing BMO norm, which results in BMO functions only being defined up to additive constants.

Theorems~\ref{thm:BMOtoCar} and \ref{thm:CarToBMO} generalize, on the one hand, the unweighted vector-valued results from~\cite{H:Wavelets} and, on the other, the weighted but scalar-valued theorems from~\cite{HSV}. More precisely, the case \(w\equiv\varrho\equiv 1\) of the above theorems, for a more restricted class of wavelets, is contained in \cite{H:Wavelets}, Proposition~4.1; in this case, the full range \(p\in(1,\infty)\) is admissible in Theorem~\ref{thm:BMOtoCar}. With \(X=\C\) and \(p=2\), the above results essentially reduce to \cite{HSV}, Theorems~A and~B, but a different sense of convergence (involving an appropriate weighted version of the \(H^1\)--BMO-duality) of the series \(\sum_{J\in\mathscr{D}}a_J\psi_J\) was used there.

For \(X=\C\), the Carleson norms have equivalent non-probabilistic expressions thanks to Hin\v{c}in's inequality:
\begin{equation*}
   \begin{split}
  &\Norm{\{c_J\}_{J\in\mathscr{D}}}{C_{\varrho}^p(w;\C)}  \\
  &\eqsim \sup_I\frac{1}{\varrho(\abs{I})}\Big(\frac{1}{w(I)}\int_I\Big[\sum_{\ontop{J\in\mathscr{D}}{J\subseteq I}}
       \abs{c_J}^2\Big(\frac{\abs{J}}{w(J)}\Big)^{2/p'}\frac{1_J(x)}{\abs{J}}\Big]^{p/2}\ud x\Big)^{1/p},
\end{split}
\end{equation*}
with equality for \(p=2\), as already mentioned.
When \(p\neq 2\), both the definition of these norms and their appearance in the above theorems appear to be new even in the scalar case.
Other variants of \(p\)-dependent Carleson norms have been recently used in \cite{HMP,HW:BMO}.

A word on the organization of the paper: the following two sections contain preliminary material, after which Theorems~\ref{thm:BMOtoCar} and~\ref{thm:CarToBMO} are proved in the last two sections.

\subsection*{Acknowledgement}
A major part of the research was carried out during T.H.'s visit to the Instituto de Matem\'atica Aplicada del Litoral, Santa Fe, in November 2007. He wants to thank the colleagues in Santa Fe for their kind hospitality, and the institute for financial support. T.H. was also supported by the Academy of Finland through the projects ``Stochastic and harmonic analysis, interactions and applications'' and ``Vector-valued singular integrals''.

\section{Preliminaries}

\subsection{Vector-valued random series}
Due to the very definition of our Carleson spaces \(C_{\varrho}^p(w;X)\), it is clear that some knowledge on how to handle the vector-valued random series
\begin{equation*}
  \Big(\Exp\BNorm{\sum_{J\subseteq I}\radem_J\xi_J}{X}^p\Big)^{1/p}
\end{equation*}
will be needed in proving the two theorems. In fact, there are only a few basic tricks which we shall employ, and they are recalled in this section.

The basic estimate is Kahane's contraction principle (\cite{Kahane}, Theorem~2.5), which allows to ``pull out'' bounded scalar coefficients,
\begin{equation*}
  \Big(\Exp\BNorm{\sum_{J\subseteq I}\radem_J\lambda_J\xi_J}{X}^p\Big)^{1/p}
  \leq\Big(\Exp\BNorm{\sum_{J\subseteq I}\radem_J\xi_J}{X}^p\Big)^{1/p}
\end{equation*}
if \(\lambda_J\in[-1,1]\); with complex \(\abs{\lambda_J}\leq 1\), one gets a similar estimate with an additional factor \(2\) on the right by simply splitting to real and imaginary parts. A very particular case of this estimate, corresponding to coefficients which are zero except for one, is the fact that the norm of the random series dominates the norm of any of the vectors appearing in it.

A somewhat deeper result, which relies on the UMD property of the space \(X\) and assumes that \(p\in(1,\infty)\), is Bourgain's vector-valued Stein inequality (\cite{Bourgain}, Lemma~8; cf.~\cite{CPSW}, Proposition~3.8), which allows to ``pull out'' averaging operators:
\begin{equation*}
\begin{split}
  &\Big(\int_I\Exp\BNorm{\sum_{J\subseteq I}\radem_J\frac{1_J(x)}{\abs{J}}\int_J f_J(y)\ud y}{X}^p\ud x\Big)^{1/p} \\
  &\leq C\Big(\int_I\Exp\BNorm{\sum_{J\subseteq I}\radem_J 1_J(x)f_J(x)}{X}^p\ud x\Big)^{1/p}.
\end{split}
\end{equation*}

Moreover, Kahane's inequality (\cite{Kahane}, Eq.~$(*)$ on p.~282) permits changing the exponent, in fact,
\begin{equation*}
   \Big(\Exp\BNorm{\sum_{J\subseteq I}\radem_J\xi_J}{X}^p\Big)^{1/p}
   \eqsim\Big(\Exp\BNorm{\sum_{J\subseteq I}\radem_J\xi_J}{X}^r\Big)^{1/r}
\end{equation*}
for all \(p,r\in[1,\infty)\). For \(X=\C\) and \(r=2\), this reduces after simplification to the classical Hin\v{c}in inequality
\begin{equation*}
   \Big(\Exp\Babs{\sum_{J\subseteq I}\radem_J\lambda_J}^p\Big)^{1/p}
   \eqsim\Big(\sum_{J\subseteq I}\abs{\lambda_J}^2\Big)^{1/2}.
\end{equation*}

\subsection{Weighted John--Nirenberg inequality}

For weighted BMO functions, the celebrated John--Nirenberg inequality takes the following form: Given \(w\in A_q\), \(q\in(1,\infty)\), and a growth function \(\varrho\) with the doubling property, the norm \(\Norm{f}{\BMO_{\varrho}(w;X)}\) is equivalent to
\begin{equation*}
  \sup_I\frac{1}{\varrho(\abs{I})}\Big(\frac{1}{w(I)}
   \int_I\Norm{f(x)-\ave{f}_I}{X}^p w^{1-p}(x)\ud x\Big)^{1/p}
\end{equation*}
for all \(p\in(1,q']\); clearly \(p=1\) corresponds to the original norm of this space.

This was first proved in the case \(\varrho\equiv 1\) by Muckenhoupt and Wheeden~\cite{MW}, and then extended to the growth function case by Morvidone~\cite{Morvidone}. Their results are stated in the case \(X=\C\), but an inspection of the proofs reveals that they immediately generalize to the vector-valued context.

\section{Wavelets in weighted Bochner spaces}

Before studying the wavelet expansions of vector-valued BMO functions, we need some results in the \(L^p\) spaces for \(p\in(1,\infty)\). These will be collected in this section. Let us note that the unweighted case has been considered before by Kaiser and Weis~\cite{KW:Wavelets}; the general treatment here is based on similar ideas but does not presuppose any knowledge of their results. The roughness of the wavelets is the same as in \cite{Hernandez}, Section~6.4, where the unweighted, scalar-valued case is treated by maximal function techniques.

\begin{definition}
A function \(K(x,y)\) defined for \(x,y\in\R\times\R\) with \(x\neq y\) is called a standard kernel provided that
\begin{equation*}
  \abs{K(x,y)}\leq C\frac{1}{\abs{x-y}},
\end{equation*}
and for some \(\delta>0\),
\begin{equation*}
  \abs{K(x,y)-K(x',y)}+\abs{K(y,x)-K(y,x')}
  \leq C\frac{\abs{x-x'}^{\delta}}{\abs{x-y}^{1+\delta}}.
\end{equation*}
\end{definition}

Some results concerning Calder\'on--Zygmund operators of the form
\begin{equation*}
  Tf(x)=\int_{\R}K(x,y)f(y)\ud y,
\end{equation*}
for \(x\notin\supp f\), will be needed.
It is convenient to formulate Figiel's \(T1\) theorem~\cite{Figiel} in the following form:

\begin{theorem}
Let \(K\) be a standard kernel and \(T\in\bddlin(L^2(\R))\). Then for every UMD space \(X\) and \(p\in(1,\infty)\), \(T\) is also bounded on \(L^p(\R;X)\).
\end{theorem}

The following extrapolation result can be extracted out of the more general statements in Theorems~1.2 and~1.3 of \cite{RRT}.

\begin{theorem}
Let \(X\) be a Banach space, \(K\) be a standard kernel and for some \(q\in(1,\infty)\),  \(T\in\bddlin(L^q(\R;X))\). Then for all \(p\in(1,\infty)\) and all \(w\in A_p\), there holds \(T\in\bddlin(L^p(w;X))\).
\end{theorem}

The two results obviously imply:

\begin{corollary}
Let \(K\) be a standard kernel and \(T\in\bddlin(L^2(\R))\). Then for every UMD space \(X\), every \(p\in(1,\infty)\) and all \(w\in A_p\), \(T\) is also bounded on \(L^p(w;X)\).
\end{corollary}

In all of the results quoted above, the bound on the norm of \(T\) only depends on the constants implicit in the assumptions, and the dependence is uniform in the sense that a family of operators verifying the assumptions with uniformly bounded constants will also satisfy the conclusions with uniform norm bounds.

Following \cite{Hernandez}, Eq. (2.12) of Chapter~6, let \(\mathscr{R}^0:=\bigcup_{\varepsilon>0}\Psi^{2+\varepsilon}_{1+\varepsilon}\).

\begin{lemma}
Let \(\phi,\psi\in\mathscr{R}^0\) and \(\abs{a_{jk}}\leq 1\). Then
\begin{equation*}
   K(x,y):=\sum_{j,k\in\Z}a_{jk}2^j\phi(2^jx-k)\psi(2^jy-k)
\end{equation*}
is a standard kernel.
\end{lemma}

\begin{proof}
Without loss of generality, \(x<y\) and \(0<\abs{x'-x}<\abs{x-y}/2\).
\begin{equation*}\begin{split}
  \abs{K(x,y)}
  &\lesssim\sum_j 2^j\sum_k(1+\abs{2^jx-k})^{-1-\varepsilon}(1+\abs{2^j y-k})^{-1-\varepsilon} \\
  &\lesssim\sum_j 2^j\Big(\sum_{k\leq(x+y)2^{j-1}}(1+\abs{2^jx-k})^{-1-\varepsilon}\cdot(1+2^j\abs{x-y})^{-1-\varepsilon}\\
  &\phantom{\lesssim\sum_j 2^j\Big(}
                         +\sum_{k>(x+y)2^{j-1}}(1+2^j\abs{x-y})^{-1-\varepsilon}\cdot(1+\abs{2^jy-k})^{-1-\varepsilon}\Big)\\
  &\lesssim\sum_j 2^j(1+2^j\abs{x-y})^{-1-\varepsilon}\\
  &\lesssim\sum_{j:2^j\abs{x-y}\leq 1} 2^j+\sum_{j:2^j\abs{x-y}>1} 2^{-j\varepsilon}\abs{x-y}^{-1-\varepsilon}
   \lesssim\frac{1}{\abs{x-y}}.
\end{split}\end{equation*}

Consider next the difference
\begin{equation*}
  \abs{K(y,x)-K(y,x')}
  \leq\sum_j\sum_k 2^j(1+\abs{2^j y-k})^{-2-\varepsilon}\abs{\psi(2^j x-k)-\psi(2^j x'-k)};
\end{equation*}
the other type of difference will have a similar bound by symmetry.
By the two obvious estimates either applying the mean value theorem or the triangle inequality,
\begin{equation*}
  \abs{\psi(2^jx-k)-\psi(2^jx'-k)}\lesssim
   \begin{cases}
     2^j\abs{x-x'}\big(1+\dist(k,2^j[x',x])\big)^{-1-\varepsilon}, &\\
     \big(1+\dist(k,2^j[x,x'])\big)^{-2-\varepsilon}. &\\
  \end{cases}
\end{equation*}

Let \(A>0\) be an auxiliary number to be chosen. The part of the sum with small \(j\) is estimates as follows:
\begin{equation}\label{eq:jSmall}\begin{split}
  &\sum_{2^j\leq A}2^j\sum_k 2^j\abs{x-x'}\big(1+\dist(k,2^j[x',x])\big)^{-1-\varepsilon}
    (1+\abs{2^jy-k})^{-1-\varepsilon} \\
  &\lesssim\sum_{2^j\leq A}2^{2j}\abs{x-x'}\Big(\sum_{k\leq(x+y)2^{j-1}}\big(1+\dist(k,2^j[x',x])\big)^{-1-\varepsilon}
    (2^j\abs{x-y})^{-1-\varepsilon} \\
  &\phantom{\lesssim\sum_{2^j\leq A}2^{2j}\abs{x-x'}\Big(}
    +\sum_{k>(x+y)2^{j-1}}\big(2^j\abs{x-y}\big)^{-1-\varepsilon}(1+\abs{2^jy-k})^{-1-\varepsilon}\Big) \\
  &\lesssim\sum_{2^j\leq A}2^{j(1-\varepsilon)}\frac{\abs{x-x'}}{\abs{x-y}^{1+\varepsilon}}\big(1+2^j\abs{x-x'}\big) \\
  &\lesssim A^{1-\varepsilon}\frac{\abs{x-x'}}{\abs{x-y}^{1+\varepsilon}}\big(1+A\abs{x-x'}\big).
\end{split}\end{equation}
As for large \(j\), there holds
\begin{equation}\label{eq:jLarge}\begin{split}
  &\sum_{2^j>A}2^j\sum_k\big(1+\dist(k,2^j[x',x])\big)^{-2-\varepsilon}(1+\abs{2^jy-k})^{-2-\varepsilon} \\
  &\lesssim\sum_{2^j>A}2^j\Big(\sum_{k\leq(x+y)2^{j-1}}\big(1+\dist(k,2^j[x',x])\big)^{-2-\varepsilon}
         (2^j\abs{x-y})^{-2-\varepsilon} \\
  &\phantom{\lesssim\sum_{2^j>A}2^j\Big(}
        +\sum_{k>(x+y)2^{j-1}}(2^j\abs{x-y})^{-2-\varepsilon}
         (1+\abs{2^jy-k})^{-2-\varepsilon}\Big) \\
  &\lesssim\sum_{2^j>A}2^{-(1+\varepsilon)j}\abs{x-y}^{-2-\varepsilon}\big(1+2^j\abs{x-x'}\big) \\
  &\lesssim\frac{A^{-1-\varepsilon}}{\abs{x-y}^{2+\varepsilon}}\big(1+A\abs{x-x'}\big).
\end{split}\end{equation}

Requiring the equality of the two upper bounds and solving for \(A\) gives \(A=\abs{x-x'}^{-1/2}\abs{x-y}^{-1/2}\). Then \(A\abs{x-x'}=\big(\abs{x-x'}/\abs{x-y}\big)^{1/2}\leq 1\), so the upper bound in both~\eqref{eq:jSmall} and~\eqref{eq:jLarge} becomes
\begin{equation*}
  \frac{\abs{x-x'}^{(1+\varepsilon)/2}}{\abs{x-y}^{(3+\varepsilon)/2}},
\end{equation*}
and hence the claim is proved with \(\delta=(1+\varepsilon)/2\).
\end{proof}

\begin{theorem}\label{thm:waveletLpAp}
Let \(\psi,\phi\in\mathscr{R}^0\) be orthonormal wavelets. Let \(X\) be a UMD space, \(1<p<\infty\), and \(w\in A_p\). Then for all \(f\in L^p(w;X)\),
\begin{equation*}\begin{split}
  \Norm{f}{L^p(w;X)}
  &\eqsim\Big(\Exp\BNorm{\sum_{J\in\mathscr{D}}\radem_J\pair{f}{\psi_J}\psi_J}{L^p(w,X)}^p\Big)^{1/p} \\
  &\eqsim\Big(\Exp\BNorm{\sum_{J\in\mathscr{D}}\radem_J\pair{f}{\psi_J}\phi_J}{L^p(w,X)}^p\Big)^{1/p} \\
  &\eqsim\Big(\Exp\BNorm{\sum_{J\in\mathscr{D}}\radem_J\pair{f}{\psi_J}\frac{1_J}{\abs{J}^{1/2}}}{L^p(w;X)}^p\Big)^{1/p}.
\end{split}\end{equation*}
\end{theorem}

\begin{proof}
The first and second comparison follow from the fact that the operators of the form
\begin{equation*}
  f\mapsto\sum_{J\in\mathscr{D}}\radem_J\pair{f}{\psi_J}\phi_J
\end{equation*}
are uniformly bounded Calder\'on--Zygmund operators.

As for the last comparison, it has been shown in \cite{H:Wavelets} (see the proof on p.~134) that there is finite collection \(\Phi\) of orthonormal wavelets \(\phi\in\mathscr{R}^0\) (in fact even infinitely regular) such that
\begin{equation*}
  \frac{1_{J}}{\abs{J}^{1/2}}\leq C\sum_{\phi\in\Phi}\abs{\phi_J}.
\end{equation*}
Hence, by the contraction principle,
\begin{equation*}
\begin{split}
  \Big(\Exp\BNorm{\sum_{J\in\mathscr{D}}\radem_J\pair{f}{\psi_J}\frac{1_J}{\abs{J}^{1/2}}}{L^p(w;X)}^p\Big)^{1/p}
  &\lesssim\Big(\sum_{\phi\in\Phi}\Exp\BNorm{\sum_{J\in\mathscr{D}}\radem_J\pair{f}{\psi_J}\phi_J}{L^p(w;X)}^p\Big)^{1/p} \\
  &\lesssim\Norm{f}{L^p(w;X)}
\end{split}
\end{equation*}
by the part already proved.

As for the other direction, let \(\phi\in\mathscr{R}^0\) be an orthonormal wavelet with compact support, and fix some \(I_1\subseteq I_0:=[0,1)\) such that \(\phi_{I_1}\) is supported in \(I_0\). Let \(\Lambda:\mathscr{D}\to\mathscr{D}\) be the mapping \(J=\inf J+\abs{J}\cdot I_0\mapsto \inf J+\abs{J}\cdot I_1\). Then \(\{\phi_{\Lambda(J)}\}_{J\in\mathscr{D}}\) is an orthonormal (incomplete) system in \(L^2(\R)\). Because of the support property and regularity, for some \(c\) there holds \(1_J/\abs{J}^{1/2}\geq c\abs{\phi_{\Lambda(J)}}\). Hence by the contraction principle,
\begin{equation*}
\begin{split}
 \Big(\Exp\BNorm{\sum_{J\in\mathscr{D}}\radem_J\pair{f}{\psi_J}\frac{1_J}{\abs{J}^{1/2}}}{L^p(w;X)}^p\Big)^{1/p}
  &\gtrsim\Big(\Exp\BNorm{\sum_{J\in\mathscr{D}}\radem_J\pair{f}{\psi_J}\phi_{\Lambda(J)}}{L^p(w;X)}^p\Big)^{1/p}\\
  &\gtrsim\Norm{f}{L^p(w,X)}
\end{split}
\end{equation*}
since the mapping
\begin{equation*}\begin{split}
  \sum_{J\in\mathscr{D}}a_J\phi_{\Lambda(J)}\mapsto\sum_{J\in\mathscr{D}}a_J\psi_J,\qquad\text{or}\qquad
  \sum_{I\in\mathscr{D}}a_I\phi_{I}\mapsto\sum_{I\in\Lambda(\mathscr{D})}a_I\psi_{\Lambda^{-1}(I)},  
\end{split}\end{equation*}
is a bounded Calder\'on--Zygmund operator.
\end{proof}

\section{BMO implies Carleson}

We now turn to the proof of Theorem~\ref{thm:BMOtoCar}. Fix a function \(f\in\BMO_{\varrho}(w;X)\) and a finite interval \(I\subset\R\). For \(\ell\in\Z_+\), let \(I_{\ell}:=2^{\ell}I\) (the interval concentric with $I$ and $2^{\ell}$ times as long), \(f_1:=(f-\ave{f}_I)1_{2I}\), and \(f_{\ell}:=(f-\ave{f}_I)1_{I_{\ell}\setminus I_{\ell-1}}\) for \(\ell\geq 2\). Then \(f=\ave{f}_I+\sum_{\ell=1}^{\infty}f_{\ell}\) and \(\pair{f}{\psi_J}=\sum_{\ell=1}^{\infty}\pair{f_{\ell}}{\psi_J}\), since \(\psi_J\) has a vanishing integral.

Consider first \(\ell\geq 2\) fixed. Below, we abbreviate the summation condition \(J\in\mathscr{D},\ J\subseteq I\) to \(J\subseteq I\), with the implicit understanding that \(J\) is always a dyadic interval. The estimation starts with
\begin{equation*}\begin{split}
  &\Big(\int\Exp\BNorm{\sum_{J\subseteq I}\radem_J\pair{f_{\ell}}{\psi_J}
    \Big(\frac{\abs{J}}{w(J)}\Big)^{1/p'}\frac{1_J(x)}{\abs{J}^{1/2}}}{X}^p\ud x\Big)^{1/p}\\
  &=\Big(\int\Exp\BNorm{\int f_{\ell}(y)\sum_{J\subseteq I}\radem_J\psi_J(y)1_{I_{\ell}\setminus I_{\ell-1}}(y)
    \Big(\frac{\abs{J}}{w(J)}\Big)^{1/p'}\frac{1_J(x)}{\abs{J}^{1/2}}\ud y}{X}^p\ud x\Big)^{1/p}\\
  &\lesssim\Big(\int\Big\{\int \Norm{f_{\ell}(y)}{X}
    \Big[\sum_{J\subseteq I}\abs{\psi_J(y)}^2
    \Big(\frac{\abs{J}}{w(J)}\Big)^{2/p'}\frac{1_J(x)}{\abs{J}}\Big]^{1/2}\ud y\Big\}^p\ud x\Big)^{1/p}.
\end{split}\end{equation*}
Next, for \(y\in I_{\ell}\setminus I_{\ell-1}\) where \(f_{\ell}\) is supported,
\begin{equation*}\begin{split}
  &\sum_{J\subseteq I}\abs{\psi_J(y)}^2
    \Big(\frac{\abs{J}}{w(J)}\Big)^{2/p'}\frac{1_J(x)}{\abs{J}} \\
  &\lesssim\sum_{J\subseteq I}\frac{1}{\abs{J}}\Big(\frac{\abs{J}}{\abs{I_{\ell}}}\Big)^{4}
     \Big(\frac{\abs{J}}{w(J)}\Big)^{2/p'}\frac{1_J(x)}{\abs{J}}     
  =\frac{1}{\abs{I_{\ell}}^4}\sum_{J\subseteq I}\frac{\abs{J}^{2+2/p'}1_J(x)}{w(J)^{2/p'}} \\
  &\lesssim\frac{1}{\abs{I_{\ell}}^4}\sum_{J\subseteq I}\frac{\abs{J}^{2+2/p'}1_J(x)}{w(I)^{2/p'}}
    \Big(\frac{\abs{I}}{\abs{J}}\Big)^{2q/p'}
  =\frac{1}{\abs{I_{\ell}}^4}\frac{\abs{I}^{2q/p'}}{w(I)^{2/p'}}
    \sum_{J\subseteq I}\abs{J}^{2(1+1/p'-q/p')}1_J(x) \\
  &\lesssim\frac{1}{\abs{I_{\ell}}^4}\frac{\abs{I}^{2q/p'}}{w(I)^{2/p'}}
    \abs{I}^{2(1+1/p'-q/p')}1_I(x) 
   =\Big(\frac{\abs{I}}{\abs{I_{\ell}}}\Big)^4\Big(\frac{\abs{I}}{w(I)}\Big)^{2/p'}\frac{1_I(x)}{\abs{I}^2},
\end{split}\end{equation*}
where the three inequalities were applications of the pointwise bound for \(\psi\in\mathscr{R}^0\), the estimate \(w(I)/w(J)\leq C\big(\abs{I}/\abs{J}\big)^q\) for \(w\in A^q\), and finally the sum of a geometric progression where \(q\leq p'<p'+1\).

Substituting back,
\begin{equation}\label{eq:thmAmain}\begin{split}
  &\Big(\int\Exp\BNorm{\sum_{J\subseteq I}\radem_J\pair{f_{\ell}}{\psi_J}
    \Big(\frac{\abs{J}}{w(J)}\Big)^{1/p'}\frac{1_J(x)}{\abs{J}^{1/2}}}{X}^p\ud x\Big)^{1/p}\\
  &\lesssim\Big(\int\Big\{\int \Norm{f_{\ell}(y)}{X}
    \Big(\frac{\abs{I}}{\abs{I_{\ell}}}\Big)^2\Big(\frac{\abs{I}}{w(I)}\Big)^{1/p'}\frac{1_I(x)}{\abs{I}}\ud y\Big\}^p\ud x\Big)^{1/p} \\
  &= \frac{2^{-2\ell}}{w(I)^{1/p'}}\int\Norm{f_{\ell}(y)}{X}\ud y
  \leq \frac{2^{-2\ell}}{w(I)^{1/p'}}\int_{2^{\ell}I}\Norm{f(y)-\ave{f}_{I}}{X}\ud y,
\end{split}\end{equation}
where the equality involved computing the trivial integration in \(x\).

The following lemma is needed:

\begin{lemma}
If \(\Norm{f}{\BMO_{\varrho}(w;X)}\leq 1\), then
\begin{equation*}
  \int_{2^{\ell}I}\Norm{f(y)-\ave{f}_{I}}{X}\ud y
  \lesssim\sum_{k=1}^{\ell}2^{\ell-k}w(2^k I)\varrho(2^k\abs{I}).
\end{equation*}
\end{lemma}

\begin{proof}
\begin{equation*}\begin{split}
  &\int_{2^{\ell}I}\Norm{f(y)-\ave{f}_{I}}{X}\ud y \\
  &\leq\int_{2^{\ell}I}\Norm{f(y)-\ave{f}_{2^{\ell}I}}{X}\ud y
  +\abs{2^{\ell}I}\sum_{k=1}^{\ell}\Norm{\ave{f}_{2^k I}-\ave{f}_{2^{k-1}I}}{X}
\end{split}\end{equation*}
The first term is bounded by \(w(2^{\ell}I)\varrho(2^{\ell}\abs{I})\), while the \(k\)th term in the sum has the estimate
\begin{equation*}\begin{split}
  &\Norm{\ave{f}_{2^k I}-\ave{f}_{2^{k-1}I}}{X}
  \leq\frac{1}{2^{k-1}\abs{I}}\int_{2^{k-1}I}\Norm{\ave{f}_{2^k I}-f(y)}{X}\ud y \\
  &\leq\frac{1}{2^{k-1}\abs{I}}\int_{2^k I}\Norm{\ave{f}_{2^k I}-f(y)}{X}\ud y\leq w(2^k I)\varrho(2^k\abs{I}).
\end{split}\end{equation*}
The assertion follows from the combination of these two estimates
\end{proof}

Continuing from~\eqref{eq:thmAmain} and summing over all \(\ell\geq 2\), it follows that
\begin{equation*}\begin{split}
  &\sum_{\ell=2}^{\infty}\frac{2^{-2\ell}}{w(I)^{1/p'}}\int_{2^{\ell}I}\Norm{f(y)-\ave{f}_{I}}{X}\ud y \\
  &\lesssim\sum_{\ell=2}^{\infty}\frac{2^{-2\ell}}{w(I)^{1/p'}}\sum_{k=1}^{\ell}2^{\ell-k}w(2^k I)\varrho(2^k\abs{I}) \\
  &\lesssim w(I)^{-1/p'}\sum_{k=1}^{\infty}2^{-2k}w(2^k I)\varrho(2^k\abs{I}) \\
  &\lesssim w(I)^{1/p}\sum_{k=1}^{\infty}2^{-2k} 2^{qk}\varrho(2^k\abs{I}) \\
  &\lesssim w(I)^{1/p}\abs{I}^{2-q}\int_{\abs{I}}^{\infty}\frac{\varrho(s)\ud s}{s^{3-q}}=w(I)^{1/p}\eta(\abs{I}).
\end{split}\end{equation*}
This is the desired estimate for the part considered, and it remains to treat \(\ell=1\).

\begin{lemma}
\begin{equation*}
  \Big(\frac{\abs{J}}{w(J)}\Big)^{1/p'}\leq\frac{1}{\abs{J}}\int_J\frac{\ud y}{w(y)^{1/p'}}.
\end{equation*}
\end{lemma}

\begin{proof}
The claim is equivalent to \(\abs{J}^{1+1/p'}\leq w(J)^{1/p'}w^{-1/p'}(J)\). This follows from H\"older's inequality:
\begin{equation*}
  \abs{J}=\int_J w^{\alpha}w^{-\alpha}
  \leq\Big(\int_J w^{\alpha q}\Big)^{1/q}\Big(\int_J w^{-\alpha q'}\Big)^{1/q'}
\end{equation*}
with \(q=1/\alpha=p'+1\), hence \(q'=1+1/p'\).
\end{proof}

By the contraction principle and Stein's inequality,
\begin{equation*}\begin{split}
  &\Big(\int\Exp\BNorm{\sum_{J\subseteq I}\radem_J\pair{f_1}{\psi_J}
     \Big(\frac{\abs{J}}{w(J)}\Big)^{1/p'}\frac{1_J(x)}{\abs{J}^{1/2}}}{X}^p\ud x\Big)^{1/p} \\
  &\lesssim\Big(\int\Exp\BNorm{\sum_{J\subseteq I}\radem_J\pair{f_1}{\psi_J}
     \frac{1}{\abs{J}}\int_J\frac{\ud y}{w(y)^{1/p'}}\frac{1_J(x)}{\abs{J}^{1/2}}}{X}^p\ud x\Big)^{1/p} \\
  &\lesssim\Big(\int\Exp\BNorm{\sum_{J\subseteq I}\radem_J\pair{f_1}{\psi_J}
     \frac{1}{w(x)^{1/p'}}\frac{1_J(x)}{\abs{J}^{1/2}}}{X}^p\ud x\Big)^{1/p} \\
  &=\Big(\int\Exp\BNorm{\sum_{J\subseteq I}\radem_J\pair{f_1}{\psi_J}
     \frac{1_J(x)}{\abs{J}^{1/2}}}{X}^p\frac{\ud x}{w^{p-1}(x)}\Big)^{1/p}. \\
\end{split}\end{equation*}
Since \(w\in A_q\subseteq A_{p'}\) (recalling that \(q\leq p'\)), we have \(w^{1-p}\in A_p\), and the estimate further continues with
\begin{equation*}\begin{split}
  &\lesssim\Norm{f_1}{L^p(w^{1-p};X)}
  =\Big(\int_{2I}\Norm{f(x)-\ave{f}_I}{X}^p\frac{\ud x}{w^{p-1}(y)}\Big)^{1/p} \\
  &\lesssim\Big(\int_{2I}\Norm{f(x)-\ave{f}_{2I}}{X}^p\frac{\ud x}{w^{p-1}(y)}\Big)^{1/p}
    +\big(w^{1-p}(2I)\big)^{1/p}\Norm{\ave{f}_{2I}-\ave{f}_I}{X}.
\end{split}\end{equation*}
The first term is bounded by \(C w(2I)^{1/p}\varrho(2\abs{I})\lesssim w(I)^{1/p}\varrho(\abs{I})\) by Morvidone's weighted John--Nirenberg inequality. By the defining inequality of \(w^{1-p}\in A_p\), there holds \(w^{1-p}(2I)\lesssim w(2I)^{1-p}\abs{2I}^p\), whereas \(\Norm{\ave{f}_{2I}-\ave{f}_I}{X}\lesssim w(2I)\varrho(2\abs{I})\). Hence the bound \(w(I)^{1/p}\varrho(\abs{I})\) is valid also for this term after some simplification.

Altogether, we have shown that
\begin{equation*}
\begin{split}
  \sum_{\ell=1}^{\infty}
  &\Big(\int\Exp\BNorm{\sum_{J\subseteq I}\radem_J\pair{f_{\ell}}{\psi_J}
     \Big(\frac{\abs{J}}{w(J)}\Big)^{1/p'}\frac{1_J(x)}{\abs{J}^{1/2}}}{X}^p\ud x\Big)^{1/p} \\
  &\lesssim w(I)^{1/p}\eta(\abs{I})\Norm{f}{\BMO_{\varrho}(w;X)},
\end{split}
\end{equation*}
and this completes the proof of Theorem~\ref{thm:BMOtoCar} since \(f\) and \(I\) were arbitrary.

\section{Carleson implies BMO}

We now turn to the proof of Theorem~\ref{thm:CarToBMO}.
Let \(\{a_J\}_{J\in\mathscr{D}}\in C_{\varrho}^p(w;X)\), without loss of generality with norm at most \(1\).
Fix a finite interval \(I\subset\R\), and consider the collections of dyadic intervals
\begin{equation*}\begin{split}
  \mathscr{J}_1 &:=\{J\in\mathscr{D};2\abs{J}>\abs{I}\}, \\
  \mathscr{J}_2 &:=\{J\in\mathscr{D};2\abs{J}\leq\abs{I}, 2J\cap 2I=\varnothing \}, \\
  \mathscr{J}_3 &:=\{J\in\mathscr{D};2\abs{J}\leq\abs{I}, 2J\cap 2I\neq\varnothing \},
\end{split}\end{equation*}
and the a priori formal series
\begin{equation*}
  f_1(x):=\sum_{J\in\mathscr{J}_1}a_J[\psi_J(x)-\psi_J(x_I)],\qquad
  f_i(x):=\sum_{J\in\mathscr{J}_i}a_J\psi_J(x),\quad i=2,3,
\end{equation*}
where $x_I$ is the centre of the interval $I$.

In order to prove Theorem~\ref{thm:CarToBMO}, we want to show that \(f_I:=f_1+f_2+f_3\) converges in the asserted sense, and moreover
\begin{equation}\label{eq:toProve}
  \int_I\Norm{f_I(x)}{X}\ud x\lesssim\eta(\abs{I})w(I).
\end{equation}

Suppose for the moment that this is already done.

\begin{lemma}\label{lem:diffIntervals}
Given two intervals $I\subset I'$, the function $f_{I'}-f_I$ is constant on $I$.
\end{lemma}

\begin{proof}
Let $\mathscr{J}_i$, $f_i$ be as above, and $\mathscr{J}_i'$, $f_i'$ denote the corresponding collections and functions related to $I'$ instead of $I$. Thus
\begin{equation*}
  f_I=f_1+f_2+f_3,\qquad f_{I'}=f_1'+f_2'+f_3'.
\end{equation*}
Now clearly $\mathscr{J}_1'\subseteq\mathscr{J}_1$ and $\mathscr{J}_2\cup\mathscr{J}_3\subseteq\mathscr{J}_2'\cup\mathscr{J}_3'$, so we can write
\begin{equation*}
\begin{split}
  f_1'-f_1 &=\sum_{J\in\mathscr{J}_1'}a_J[\psi_J-\psi_J(x_{I'})]-\sum_{J\in\mathscr{J}_1}a_J[\psi_J-\psi_J(x_{I})] \\
    &=\sum_{J\in\mathscr{J}_1'}a_J[\psi_J(x_I)-\psi_J(x_{I'})]
        -\sum_{J\in\mathscr{J}_1\setminus\mathscr{J}_1'}a_J[\psi_J-\psi_J(x_{I})],
\end{split}
\end{equation*}
and
\begin{equation*}
\begin{split}
  (f_2'+f_3')-(f_2+f_3) &=\sum_{J\in\mathscr{J}_2'\cup\mathscr{J}_3'}a_J\psi_J-\sum_{J\in\mathscr{J}_2\cup\mathscr{J}_3}a_J\psi_J\\
    &=\sum_{J\in(\mathscr{J}_2'\cup\mathscr{J}_3')\setminus(\mathscr{J}_2\cup\mathscr{J}_3)}a_J\psi_J.
\end{split}
\end{equation*}
Observing that $(\mathscr{J}_2'\cup\mathscr{J}_3')\setminus(\mathscr{J}_2\cup\mathscr{J}_3)=\mathscr{J}_1\setminus\mathscr{J}_1'$, it follows upon summing up that
\begin{equation*}
  f_{I'}-f_I=\sum_{J\in\mathscr{J}_1'}a_J[\psi_J(x_I)-\psi_J(x_{I'})]
    +\sum_{J\in\mathscr{J}_1\setminus\mathscr{J}_1'}a_J \psi_J(x_{I})
  =\text{constant}.
\end{equation*}
Note that the convergence of this $X$-valued series follows from the fact that it is a sum of the convergent function series above, and a series of $X$-valued constant functions converges in $L^s(\R;X)$ if and only if it converges in $X$.
\end{proof}

Then consider an increasing sequence of intervals \(I_1\subset I_2\subset\ldots\to\R\). By Lemma~\ref{lem:diffIntervals}, there are constants \(\xi_k\in X\) such that \(f_{I_k}|_{I_1}=f_{I_1}+\xi_k\). Then \(f_{I_k}|_{I_{k-1}}-\xi_k = f_{I_{k-1}}-\xi_{k-1}\) on \(I_1\), and hence on all of \(I_{k-1}\),  since \(f_{I_k}|_{I_{k-1}}-f_{I_{k-1}}\) is also a constant.
Thus
\begin{equation*}
  f(x):=f_{I_k}(x)-\xi_k\qquad\text{if }x\in I_k
\end{equation*}
gives a well-defined function on all of \(\R\). If \(I\subset\R\) is any finite interval, then \(I\subset I_k\) for some \(k\), and thus \(f|_I=f_{I_k}|_I-\xi_k=f_I+\xi_I\) for some \(\xi_I\in X\). Thus \eqref{eq:toProve} is just the BMO condition for \(f\) corresponding to the interval \(I\). It hence suffices to prove \eqref{eq:toProve}, with the asserted convergence of the series defining \(f_I\).

We first deal with \(f_1\) and \(f_2\). For them, not only does the convergence happen in a much stronger sense, but also we only need to exploit a rather weak consequence of the assumed Carleson estimate, namely the following bound for individual terms:
\begin{equation}\label{eq:individual}
  \Norm{a_J}{X}\leq\varrho(\abs{I})w(I)\,\abs{I}^{-1/2}.
\end{equation}

\begin{lemma}\label{lem:f1}
The series defining \(f_1(x)\) converges absolutely and uniformly for \(x\in I\), and the limit satisfies \(\abs{I}\cdot\Norm{f_1(x)}{X}\lesssim w(I)\eta(\abs{I})\).
\end{lemma}

\begin{proof}
Using \eqref{eq:individual} and the derivative bound for \(\psi\in\Psi_2^{2+\varepsilon}\),
\begin{equation*}\begin{split}
  \Norm{f_1(x)}{X} &\leq\sum_{J\in\mathscr{J}_1}\Norm{a_J}{X}\cdot\abs{\psi_J(x)-\psi_J(x_I)} \\
   &\lesssim\sum_{J\in\mathscr{J}_1}\varrho(\abs{J})w(J)\abs{J}^{-1/2}\cdot
     \abs{I}\abs{J}^{-3/2}\Big(1+\frac{\dist(I,J)}{\abs{J}}\Big)^{-2} \\
   &\lesssim\abs{I}\sum_{j=0}^{\infty}\varrho(2^j\abs{I})(2^j\abs{I})^{-2}
    \sum_{\abs{J}\in(2^{j-1},2^j]\abs{I}}w(J)\Big(1+\frac{\dist(I,J)}{\abs{J}}\Big)^{-2}.
\end{split}\end{equation*}
The inner sum is comparable with
\begin{equation*}
  w(2^jI)+\sum_{\ell=0}^{\infty}w(2^{j+\ell+1}I\setminus 2^{j+\ell}I)2^{-2\ell}
  \lesssim w(2^jI)+\sum_{\ell=1}^{\infty}w(I)2^{(q+j)\ell}2^{-2\ell}
  \lesssim 2^{jq}w(I),
\end{equation*}
since \(w(2^{\ell}I)/w(I)\lesssim 2^{q\ell}\) for \(w\in A_q\), and \(q<2\).
Substituting back gives
\begin{equation*}\begin{split}
  \Norm{f_1(x)}{X} &\lesssim \abs{I}w(I)\sum_{j=0}^{\infty}\varrho(2^j\abs{I})(2^j\abs{I})^{-2}2^{jq} \\
   &\lesssim\frac{w(I)}{\abs{I}}\abs{I}^{2-q}\int_{\abs{I}}^{\infty}\frac{\varrho(s)}{s^{3-q}}\ud s
    =\frac{w(I)\eta(\abs{I})}{\abs{I}},
\end{split}\end{equation*}
which completes the proof.
\end{proof}

\begin{lemma}
The series defining \(f_2(x)\) converges absolutely and uniformly for \(x\in I\), and the limit satisfies \(\abs{I}\cdot\Norm{f_2(x)}{X}\lesssim w(I)\varrho(\abs{I})\).
\end{lemma}

\begin{proof}
Using \eqref{eq:individual} and the pointwise bound for \(\psi\in\Psi_2^{2+\varepsilon}\) (\(\varepsilon=0\) suffices here),
\begin{equation*}\begin{split}
  \Norm{f_2(x)}{X} &\leq\sum_{J\in\mathscr{J}_2}\Norm{a_J}{X}\cdot\abs{\psi_J(x)} \\
  &\lesssim\sum_{J\in\mathscr{J}_2}\varrho(\abs{J})w(J)\abs{J}^{-1/2}\cdot
    \abs{J}^{-1/2}\Big(\frac{\dist(J,I)}{\abs{J}}\Big)^{-2} \\
  &\lesssim\sum_{j=1}^{\infty}\varrho(2^{-j}\abs{I})(2^{-j}\abs{I})^{-1}
    \sum_{\begin{smallmatrix}\abs{J}\in(2^{-j-1},2^{-j}]\abs{I} \\
          \dist(J,I)>2^{-1}\abs{I}\end{smallmatrix}}w(J)\Big(\frac{\dist(J,I)}{\abs{J}}\Big)^{-2}.
\end{split}\end{equation*}
Similarly to the previous proof, the inner sum is comparable with
\begin{equation*}
  \sum_{\ell=0}^{\infty}w(2^{\ell+1}I\setminus 2^{\ell}I)2^{-(\ell+j)2}
  \lesssim w(I)\sum_{\ell=0}^{\infty}2^{\ell q}2^{-2(\ell+j)}
  \lesssim 2^{-2j} w(I).
\end{equation*}
Substituting back and using the trivial bound \(\varrho(2^{-j}\abs{I})\leq\varrho(\abs{I})\), it follows that
\begin{equation*}
  \Norm{f_2(x)}{X}\lesssim\frac{\varrho(\abs{I})w(I)}{\abs{I}}\sum_{j=1}^{\infty}2^{-j},
\end{equation*}
and the claim follows.
\end{proof}

\begin{lemma}
For a sufficiently small \(s>1\), the series defining \(f_3(x)\) converges unconditionally in \(L^s(I;X)\), and the limit satisfies
\begin{equation*}
  \abs{I}^{1/s'}\Big(\int_I\Norm{f_3(x)}{X}^s\ud x\Big)^{1/s}
  \lesssim w(I)\varrho(\abs{I}).
\end{equation*}
\end{lemma}

\begin{proof}
\begin{equation*}\begin{split}
  &\Big(\int_I\Norm{f_3(x)}{X}^s\ud x\Big)^{1/s} \\
  &\lesssim\Big(\int_{\R}\Exp\BNorm{\sum_{J\in\mathscr{J}_3}
    \radem_J a_J\frac{1_J(x)}{\abs{J}^{1/2}}}{X}^s\ud x\Big)^{1/s} \\
  &\lesssim\Big(\int_{\R}\Exp\BNorm{\sum_{J\in\mathscr{J}_3}
    \radem_J a_J\Big(\frac{\abs{J}}{w(J)}\Big)^{1/p'}\frac{1_J(x)}{\abs{J}^{1/2}}
    M(1_{4I}w)^{1/p'}(x)}{X}^s\ud x\Big)^{1/s} \\
  &\lesssim\Big(\int_{\R}\Exp\BNorm{\sum_{J\in\mathscr{J}_3}
    \radem_J a_J\Big(\frac{\abs{J}}{w(J)}\Big)^{1/p'}\frac{1_J(x)}{\abs{J}^{1/2}}}{X}^p\ud x\Big)^{1/p}\\
  &\phantom{\lesssim\abs{I}^{1/s'}}\times\Big(\int_{\R}M(1_{4I}w)^{s(p-1)/(p-s)}(x)\ud x\Big)^{(p-s)/ps}
\end{split}\end{equation*}
by Theorem~\ref{thm:waveletLpAp} (the unweighted version suffices here), the pointwise bound \(w(J)/\abs{J}\leq CM(1_{4I}w)(x)\) for \(x\in I\) together with the contraction principle, and H\"older's inequality with exponent \(p/s\).

The first integral norm is bounded by \(w(I)^{1/p}\varrho(\abs{I})\) due to the Carleson condition, whereas
\begin{equation*}\begin{split}
  &\Big(\int_{\R}M(1_{4I}w)^{s(p-1)/(p-s)}(x)\ud x\Big)^{(p-s)/ps} \\
  &\lesssim\Big(\int_{4I}w^{s(p-1)/(p-s)}(x)\ud x\Big)^{(p-s)/ps} \\
  &=\Big(\frac{1}{4\abs{I}}\int_{4I}w^{s(p-1)/(p-s)}(x)\ud x\Big)^{(p-s)/s(p-1)\times(p-1)/p}(4\abs{I})^{(p-s)/ps} \\
  &\lesssim\Big(\frac{w(4I)}{4\abs{I}}\Big)^{(p-1)/p}\abs{I}^{(p-s)/ps}
   \lesssim w(I)^{1/p'}\abs{I}^{-1/s'}
\end{split}\end{equation*}
by the maximal theorem, the reverse H\"older inequality (provided that \(s\), and then \(s(p-1)/(p-s)\), is taken sufficiently close to \(1\)), and the doubling property of Muckenhoupt weights. Everything combined,
\begin{equation*}
  \Big(\int_I\Norm{f_3(x)}{X}^s\ud x\Big)^{1/s}
  \lesssim w(I)^{1/p}\varrho(\abs{I})\times w(I)^{-1/p'}\abs{I}^{-1/s'}=w(I)\varrho(\abs{I})\abs{I}^{-1/s'},
\end{equation*}
which is again the asserted bound.
\end{proof}

Clearly all the estimates in the previous lemmas were actually stronger than
\begin{equation*}
  \int_I\Norm{f_i(x)}{X}\ud x\lesssim\varrho(\abs{I})w(I);
\end{equation*}
hence the proof of \eqref{eq:toProve}, and then of Theorem~\ref{thm:CarToBMO}, is complete.


\end{document}